\documentclass[11pt]{article}
\usepackage{amscd,amssymb,stmaryrd}
\usepackage[fleqn]{amsmath}
\usepackage{graphicx}
\usepackage{subcaption}
\usepackage[utf8]{inputenc}
\usepackage[export]{adjustbox}
\usepackage{wrapfig}
\usepackage{amstext}
\usepackage{pstricks,pst-node,pst-plot,pst-coil}
\usepackage{amsthm}
\usepackage{amsmath}
\usepackage{color}
\usepackage{mathtools}
\usepackage{hyperref}
\usepackage[english]{babel}
\usepackage{booktabs}
\usepackage{mathrsfs}
\usepackage{comment}
\usepackage{marginnote}
\usepackage{enumitem}

\newcommand\norm[1]{\left\|#1\right\|}
\newcommand\abs[1]{\lvert#1\rvert}

\newcommand{\tforall}{\text{ for all }}

\theoremstyle{definition}
\newtheorem{theorem}{Theorem}[section]

\newtheorem{lemma}[theorem]{Lemma}

\newtheorem{definition}[theorem]{Definition}
\theoremstyle{remark}
\newtheorem*{remark}{Remark}

\newcommand\mbp{\mathbb{P}}
\newcommand\spt{\text{spt}}
\newcommand\tand{\quad\text{and}\quad}

\newcommand{\Omeps}{\Omega^\epsilon}
\newcommand{\mbf}[1]{\mathbf{#1}}

\DeclareMathOperator{\argmin}{argmin}

\title{Convergence of the CEM-GMsFEM for Stokes flows in heterogeneous perforated domains}
\author{Eric Chung\thanks{Department of Mathematics, The Chinese University of Hong Kong, Shatin, Hong Kong.  (\texttt{E-mail: tschung@math.cuhk.edu.hk})}, 
~ Jiuhua Hu\thanks{Department of Mathematics, Texas A\&M University, College Station, TX 77843, USA.  (\texttt{E-mail: jhu@math.tamu.edu})}, 
~ and ~ Sai-Mang Pun\thanks{Department of Mathematics, Texas A\&M University, College Station, TX 77843, USA.  (\texttt{E-mail: smpun@math.tamu.edu})}}
\date{\today} 
\begin{document}
\maketitle
\begin{abstract}
In this paper, we consider the incompressible Stokes flow problem in a perforated domain and employ the constraint energy minimizing generalized multiscale finite element method (CEM-GMsFEM) to solve this problem. The proposed method provides a flexible and systematical approach to construct crucial divergence-free multiscale basis functions for approximating the displacement field. These basis functions are constructed by solving a class of local energy minimization problems over the eigenspaces that contain local information on the heterogeneities. These multiscale basis functions are shown to have the property of exponential decay outside the corresponding local oversampling regions. 
By adapting the technique of oversampling, the spectral convergence of the method with error bounds related to the coarse mesh size is proved. 

\end{abstract}

\section{Introduction}
In physics and structural mechanics there is a wide range of applications involving perforated domains (see Figure \ref{fig:perforation} for an example of perforated domain). The perforated domain is characterized by partitioning a material into a solid portion and a pore space, referred as ``matrix" and ``pores", respectively. In the model of differential equations over porous media, the state equation is built in the matrix and the boundary conditions are imposed on the boundary of the matrix, including the boundary of the pores. A direct numerical treatment of solving differential equations on such a domain  is challenging because a fine mesh discretization is needed near the pores and this will result in a large computation. 

Many model reduction techniques for problems with perforation have been well developed in the existing literature to improve the computational efficiency. For example, in numerical upscaling methods \cite{Stochastic2004,YalchinHouMultiscale, Homo2018Stokes, hornung1996homogenization, Yosifian1997Homo, Micromechanics}, 
one typically derives upscaled media or upscaled models and solves the resulting upscaled problem globally on a coarse grid. The dimensions of the corresponding linear systems are much smaller, giving a guaranteed saving of computational cost. 
In addition, various multiscale methods for simulating multiscale problems with perforations are presented in the literature. For instance, multiscale finite element methods (MsFEM) of Crouzeix-Raviart type have been developed for elliptic problem \cite{le2014msfem} and Stokes flows \cite{feng2018crouzeix,msfemstokesp2,msfemstokesp1}. In \cite{henning2009heterogeneous}, the Heterogeneous multiscale method (HMM) is proposed to discretize the elliptic problem with perforations in a coarse grid. Recently, a class of generalized finite element methods for the elliptic problem in perforated domain \cite{brown2016multiscale} has been proposed. This type of methods is based on the idea of localized orthogonal decomposition (LOD) \cite{engwer2019efficient,Peterseim2014} and generalize the traditional finite element method to accurately resolve the multiscale problems with a cheaper cost. 

In this research, we focus on the recently-developed generalized multiscale finite element method (GMsFEM) \cite{chung2016adaptive,efendiev2013generalized}. The GMsFEM is a generalization of the classical MsFEM \cite{efendiev2009multiscale} in the sense that multiple basis functions can be systematically constructed for each coarse block. The GMsFEM consists of two stages: the offline and online stages. In the offline stage, a set of (local supported) snapshot functions are constructed, which can be used to essentially capture all fine-scale features of the solution. Then, a model reduction is performed by the use of a well-designed local spectral decomposition, and the dominant modes are chosen to be the multiscale basis functions. All these computations are done before the actual simulations of the model. In the online stage, with a given source term and boundary conditions, the multiscale basis functions obtained in the offline stage are used to approximate the solution. There are some previous works using GMsFEM for the Darcy's flow model in perforated domain \cite{chung2017onlinep,chung2016mixed}, the Stokes equation with perforation \cite{Chung2017Stokes}
as well as coupled flow and transport in perforated domains \cite{chung2018multiscaleflow}. 

In this paper, we will develop and analyze a novel multiscale method for incompressible Stokes flows in perforated domains. Our idea is motivated by the recently-developed Constraint Energy Minimizing Generalized Multiscale Finite Element Method (CEM-GMsFEM), which has achieved great success in solving elliptic problems with multiscale features \cite{CEM_elliptic,chung2018cemmixed}. This method has been applied successfully in dealing with many problems, e.g., embedded fracture model for coupled flow and mechanics problem \cite{CEM2019Flow}, poroelasticity problems \cite{fu2019computational,CEM2020Elasticity}, and wave equation \cite{chung2020computational}. CEM-GMsFEM is based on the framework of GMsFEM to design multiscale basis functions such that the convergence of the method is independent of the contrast from the heterogeneities; and the error linearly decreases with respect to coarse mesh size if oversampling parameter is appropriately chosen. 
Our approach of solving velocity has two ingredients. Firstly, we construct auxiliary multiscale basis functions by solving a local eigenvalue problem on each coarse block. The global auxiliary space is formed by extending these auxiliary basis and the auxiliary space contains the information related to the pores. Secondly, the multiscale basis is sought in a weakly divergence free space by solving a minimization problem in an oversampling domain. The impose of weakly divergence free condition on the multiscale basis enables us solving velocity solitarily. We prove in Lemma \ref{lem:exponential-decay} that the multiscale basis decay exponentially outside the local oversampling domain. This exponential decay property plays a vital role in the convergence analysis of the proposed method and justifies the use of local multiscale basis functions.

We organize the paper as follows. In Section \ref{sec: problem_setting}, we state the model problem and its variational formulation. In Section \ref{sec:construction}, we introduce auxiliary space and the construction of multiscale basis functions for pressure using relaxed constraint energy minimization. The multiscale basis functions are constructed by solving a local spectral problem. We analyze convergence results in Section \ref{sec:convergence}. Concluding remarks will be drawn in Section \ref{sec:conclusion}. 

\begin{figure}[ht]
	\centering
	\includegraphics[width=2.5in]{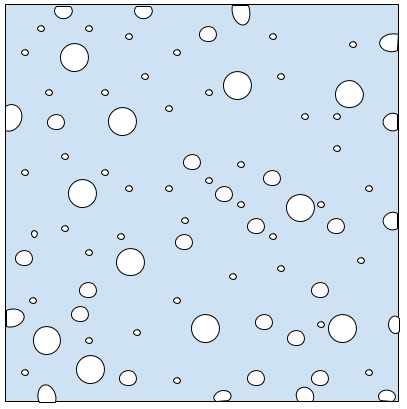}
	\caption{Illustration of a perforated domain}
	\label{fig:perforation}
\end{figure}

\section{Problem Setting}\label{sec: problem_setting}
In this section, we start with stating the Stokes flow in heterogenous perforated domains.  Then some notations and function spaces are introduced. We also introduce its corresponding variational formulation.

\subsection{Model problem}
Let $\Omega \subseteq \mathbb{R}^d$ be a bounded domain and $B_{\epsilon}$ be a set of perforations within this domain. The perforations are supposedly small and of a large number. We denote by $\Omeps := \Omega \setminus \overline{B_\epsilon}$ the perforated domain. 
Then, we consider the basic linear model for incompressible fluid mechanics, i.e, Stokes equations. Stokes problem consists of  finding  vector function $\mbf{u}\colon \Omeps \to \mathbb{R}^d$ and scalar function $p\colon \Omeps \to \mathbb{R}$ satisfying
\begin{eqnarray}
\begin{split}
- \mu\Delta \bold{u} + \nabla p & =\bold{f} \quad &\text{in } \Omega^{\epsilon}, \\
\nabla \cdot \bold{u} & = 0 \quad &\text{in } \Omega^{\epsilon}, \\
\mbf{u} & = \mbf{g} \quad &\text{on } \partial \Omeps, 
\end{split}
\label{eqn:model}
\end{eqnarray}
where the vector filed $\mbf{f} \colon \Omeps \to \mathbb{R}^d$ is the body force acting on the fluid, $\mbf{u}$ can be interpreted as the velocity of an incompressible fluid motion, $p$ is the associated pressure, and the constant $\mu$ is the viscosity coefficient fluid. 
For the sake of simplicity, we only consider homogeneous Dirichlet boundary for the velocity, i.e., $\mbf{g}=0$, and the viscosity constant $\mu=1$. The extension to the general viscosity constant and other types of boundary conditions  is straightforward. 
Since the pressure $p$ is uniquely defined up to a constant, we  assume that $\int_{\Omeps} p~ dx  = 0$ so that the problem has a unique solution. In this model, the primary source of the heterogeneity comes from the perforations in the computational domain; model reduction is necessary for practical simulation in this case. 

\subsection{Function spaces}
In this subsection, we clarify the notations used throughout the article. 
We write $(\cdot,\cdot)$ to denote the inner product in $L^2(\Omeps)$ and $\norm{\cdot}$ for the corresponding norm. 
We denote $L_0^2(\Omega^\epsilon)$ the subspace of $L^2(\Omeps)$ containing functions with zero mean. 
Let $H^1(\Omeps)$ be the classical Sobolev space with the norm $\norm{v}_1 := \left ( \norm{v}^2 + \norm{\nabla v}^2 \right )^{1/2}$ for any $v \in H^1(\Omeps)$ 
and $H_0^1(\Omeps)$ the subspace of functions having a vanishing trace. 
For vector-valued functions, we denote $\mathbf{L}^2(\Omeps):= \left ( L^2(\Omeps) \right )^d$ and $\mathbf{H}_0^1(\Omega^{\epsilon}) := \left ( H_0^1(\Omega^{\epsilon}) \right)^d$. We write $\left \langle \cdot, \cdot \right \rangle$ to denote the inner product in $\mathbf{L}^2(\Omeps)$. 
We also denote $\norm{\cdot}$ the norm induced by the inner product $\left \langle \cdot, \cdot \right \rangle$. 
To shorten notations, we define the spaces for the velocity field $\mathbf{u}$ and the pressure $p$ by
$$
\mathbf{V}_0:=\mathbf{H}_0^1(\Omeps) \quad \text{and} \quad Q_0:=L^2_0(\Omeps).
$$

\subsection{Variational formulation and fine-grid discretization}
In this subsection, we provide the variational formulation corresponding to the system \eqref{eqn:model}. 
We multiply the first equation and the second one with test functions from $\mathbf{V}_0$ and $Q_0$, respectively. Then, applying Green's formula and making use of the boundary condition, the associated variational formulation of Stokes equation reads: Find $(\bold{u}, p) \in \mathbf{V}_0 \times Q_0$ such that 
\begin{eqnarray}\label{solution:u}
\begin{split}
a(\bold{u},\bold{v}) - b(\bold{v},p) & =\langle \bold{f},\bold{v} \rangle \quad &\text{for all } \bold{v} \in \mathbf{V}_0, \\
b(\bold{u},q) & = 0 \quad &\text{for all } q \in Q_0,
\end{split}
\end{eqnarray}

where
$$ a(\bold{u},\bold{v}) := \int_{\Omega^{\epsilon}} \nabla \bold{u} : \nabla \bold{v} ~ dx, \quad \text{and} \quad  b(\bold{u},q) := \int_{\Omega^{\epsilon}} q \nabla \cdot \bold{u} ~ dx.$$

The well-posedness of \eqref{solution:u} can be proved (see, for example \cite[Chapter 4]{Guermond2004book}). 
Throughout this work, we denote $\norm{\cdot}_a := \sqrt{a(\cdot,\cdot)}$ the energy norm. In particular, for $\mbf{v}=(v_1,...,v_d)^T\in \mbf{V}_0$, $\norm{\mbf{v}}^2_a=\sum_{i=1}^d \norm{v_i}_a^2$.

To discretize the variational problem \eqref{solution:u}, let $\mathcal{T}^h$ be a conforming partition for the computational domain 
$\Omega^\epsilon$ with (local) grid sizes $h_{K}:=\text{diam}(K)$ for $K\in \mathcal{T}^h$ and $h:=\text{max}_{K\in \mathcal{T}^h}h_K$. We remark that $\mathcal{T}^h$ is referred to as the \textit{fine grid}. Next, let $\mathbf{V}_h$ and $Q_h$ be any conforming stable pair of finite element spaces with respect to the fine grid $\mathcal{T}^h$. 
For the coupling numerical scheme, one may use continuous Galerkin (CG) formulation: 
Find $(\mathbf{u}_h, p_h) \in \mathbf{V}_h \times Q_h$ such that 
\begin{eqnarray} \label{eqn:weak_dis}
\begin{split}
a(\mathbf{u}_h,\mathbf{v}_h) - b(\mathbf{v}_h,p_h) & =\left \langle \mathbf{f},\mathbf{v}_h \right \rangle \quad& \tforall \mathbf{v}_h \in \mathbf{V}_h, \\
b(\mathbf{u}_h,q_h) & = 0 \quad & \tforall  q_h \in Q_h. 
\end{split}
\end{eqnarray}
We remark that this classical approach will serve as a reference solution. The aim of this research is to construct a reduced system based on \eqref{eqn:weak_dis}. To this end, we introduce finite-dimensional multiscale spaces $\mathbf{V}_{\text{ms}} \subseteq \mathbf{V}_0$ and $Q_{\text{ms}} \subseteq Q_0$, whose dimensions are much smaller, for approximating the solution on some feasible coarse grid.

\section{Construction of multiscale spaces} \label{sec:construction}
In this section, we construct multiscale spaces on a coarse grid. 
Let $\mathcal{T}^H$ be a conforming partition of the computational domain $\Omega^\epsilon$ such that $\mathcal{T}^h$ is a refinement of $\mathcal{T}^H$.
We call $\mathcal{T}^H$ the \textit{coarse grid} and each element of $\mathcal{T}^H$  a coarse block. We denote  $H:=\text{max}_{K\in \mathcal{T}^H}\text{diam}(K)$ the coarse grid size. 
Let $N_c$ be the total number of (interior) vertices of $\mathcal{T}^H$ and $N$ be the total number of coarse elements. We remark that the coarse element $K \in \mathcal{T}^H$ is a closed subset (of the domain $\Omeps$) with nonempty interior and piecewise smooth boundary. 
Let $\{x_i\}_{i=1}^{N_c}$ be the set of nodes in $\mathcal{T}^H$. Figure \ref{fig:grid} illustrates the fine grid and a coarse element $K_i$. 

\begin{figure}[ht]
	\centering
	\includegraphics[width=3in]{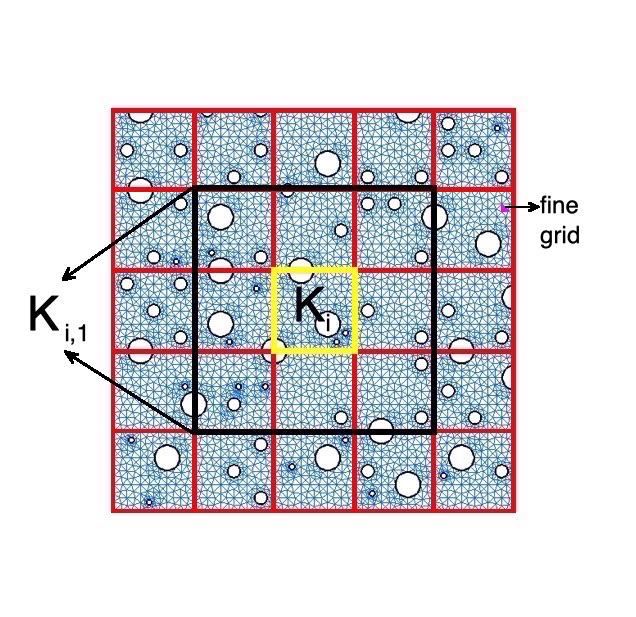}
	\caption{Illustration of the coarse grid, the fine grid, and the oversampling domain.}
	\label{fig:grid}
\end{figure}

The construction of the multiscale spaces consists of two steps. The first step is to construct auxiliary multiscale spaces using the concept of GMsFEM. Based on the auxiliary spaces, we can then construct multiscale spaces containing basis functions whose energy are minimized in some subregions of the domain. These energy-minimized basis functions will be shown to decay exponentially outside the oversampling domain, and can be used to construct a multiscale solution.

\subsection{Auxiliary space}
In this section, we begin with the construction of the auxiliary multiscale basis functions.
Let $\mathbf{V}(S)$ be the restriction of $\mathbf{V}_0$ on $S \subset \Omega^\epsilon$ and $\mathbf{V}_0(S)$ be the subspace of $\mathbf{V}(S)$, whose element is of zero trace on $\partial S$. We also define $Q_0(S) := L_0^2(S)$. 
Consider the following local spectral problem: 
Find $(\phi_j^i, \lambda_j^i)\in \mathbf{V}(K_i) \times \mathbb{R}$ such that 
\begin{eqnarray}
a_i(\phi_j^i, \mbf{v}) =  \lambda_j^i s_i(\phi_j^i, \mbf{v}) \quad \tforall  \mbf{v} \in \mathbf{V}(K_i), 
\label{eqn:spectral}
\end{eqnarray}
where $a_i(\cdot, \cdot)$ and $s_i(\cdot, \cdot)$ are defined as follows:
\begin{eqnarray}
a_i(\mbf{u},\mbf{v}) := \int_{K_i} \nabla \mbf{u} : \nabla \mbf{v}~ dx \quad \text{and} \quad s_i (\mbf{u},\mbf{v}) := \int_{K_i} \tilde{\kappa} \mbf{u} \cdot \mbf{v} ~ dx
\label{eqn:local_form}
\end{eqnarray}
for any $\mbf{u}, \mbf{v} \in \mathbf{V}(K_i)$. Here, we define $\tilde{\kappa} :=\sum_{j=1}^{N_c} \abs{\nabla \chi_j^{\text{ms}}}^2$, where $\{ \chi_j^{\text{ms}} \}_{j=1}^{N_c}$ is a set of neighborhood-wise defined partition of unity functions \cite{bm97} on the coarse grid. In particular, the function $\chi_j^{\text{ms}}$ satisfies $H \abs{\nabla \chi_j^{\text{ms}}} = O(1)$ and $0 \leq \chi_j^{\text{ms}} \leq 1$. 

Assume that the eigenvalues are arranged in ascending order such that $$0 \leq \lambda_1^i \leq \cdots \leq \lambda_{\ell_i}^i \leq \cdots$$ for each $i \in \{ 1, \cdots, N\}$. Also, we assume that the eigenfunctions satisfy the normalization condition $s_i(\phi_j^i,\phi_j^i)=1$. 
Then, we choose the first $\ell_i \in \mathbb{N}^+$ eigenfunctions and define $V^i_{\text{aux}} :=\text{span}\{ \phi^i_j: j = 1, \cdots,  \ell_i\}$. Based on these local spaces, the global auxiliary space $V_{\text{aux}}$ is defined to be 
$$V_{\text{aux}}:=\displaystyle{\bigoplus_{i=1}^N V^i_{\text{aux}}} 
\quad \text{with inner product} \quad 
s(\mbf{u}, \mbf{v}) := \sum_{i=1}^{N} s_i (\mbf{u}, \mbf{v})$$
for any $\mbf{u}, \mbf{v} \in V_{\text{aux}}$. 
Further, we define an orthogonal projection $\pi: \mbf{V}_0 \to V_{\text{aux}}$ such that 
\begin{eqnarray*}
\pi(\mbf{v}):=\sum_{i=1}^N \pi_i(\mbf{v}), \quad \text{where } \pi_i(\mbf{v}) := \sum_{j=1}^{\ell_i} s_i(\mbf{v},\phi_j^i)\phi_j^i
\end{eqnarray*}
for all  $\mbf{v} \in \mathbf{V}_0$.

\subsection{Multiscale space}
In this section, we construct multiscale basis functions based on constraint energy minimization. 
For each coarse element $K_i$, we define the oversampled region $K_{i,k_i} \subseteq \Omeps$ by enlarging $K_i$ by $k_i \in \mathbb{N}$ layer(s), i.e., 
$$ K_{i,0} := K_i, \quad K_{i,k_i} := \bigcup \{ K \in \mathcal{T}^H : K \cap K_{i,k_i -1} \neq \emptyset \} \quad \text{for } k_i = 1, 2, \cdots.$$
We call $k_i$ a parameter of oversampling related to the coarse element $K_i$. See Figure \ref{fig:grid} for an illustration of $K_{i,1}$. For simplicity, we denote $K_i^+$ a generic oversampling region related to the coarse element $K_i$ with a specific oversampling parameter $k_i$. 
Next, we define multiscale basis function possessing the property of constraint energy minimization \cite{CEM_elliptic}. In particular, for each auxiliary function $\phi_j^i \in V_{\text{aux}}$, we solve the following minimization problem: Find $\psi_{j,\text{ms}}^i \in \mbf{V}_0(K_i^+)$ such that 
\begin{eqnarray}
\psi^i_{j,\text{ms}} := \argmin \left \{ a(\psi, \psi)+s\left (\pi (\psi)-\phi^i_j, \pi(\psi)-\phi^i_j \right ):  \psi \in \mathbf{V}_0(K_{i}^{+}) \text{ and } \nabla \cdot \bold{\psi}  = 0 \right \}.
\label{eqn:original_min_relax_ms}
\end{eqnarray}
Note that problem \eqref{eqn:original_min_relax_ms} is equivalent to the local problem: Find $(\psi_{j,\text{ms}}^i , \xi_{j,\text{ms}}^i) \in \mbf{V}_0(K_i^+) \times Q_0(K_i^+)$ such that 
\begin{eqnarray}
\begin{split}
a(\psi^i_{j,\text{ms}}, v)+s\left (\pi(\psi^i_{j,\text{ms}}),\pi(v)\right ) + b(v, \xi_{j,\text{ms}}^i) &=s\left(\phi_{j}^i, \pi(v) \right) &\quad \text{for all } v\in \mathbf{V}_0 (K_i^+),\\
b(\psi^i_{j,\text{ms}},q)&=0 & \quad \text{for all } q\in Q_0(K_i^+).
\end{split}
\label{eqn:min_relax_ms}
\end{eqnarray}

Finally, for fixed parameters $k_i$ and $\ell_i$, the multiscale space $V_{\text{ms}}$ is defined by 
$$ V_{\text{ms}} := \text{span} \left \{ \psi_{j,\text{ms}}^i: 1\leq j \leq \ell_i,\ 1\leq i \leq N \right \}.$$

The multiscale basis functions can be interpreted as approximations to global multiscale basis functions $\psi_j^i\in \mathbf{V}_0$ defined by 
$$
\psi_j^i := \argmin \left \{a(\psi,\psi)+s\left (\pi(\psi)-\phi_j^i,\pi(\psi)-\phi_j^i \right ): \,\psi\in \mbf{V}_0 \text{ and } \nabla \cdot \psi = 0 \right \},
$$
which is equivalent to the following variational formulation: Find $(\psi^i_j, \xi_j^i ) \in \mathbf{V}_0 \times  Q_0$ such that 
\begin{eqnarray}
\begin{split}
a(\psi^i_j, v)+s\left (\pi(\psi^i_j),\pi(v)\right ) + b(v, \xi_j^i) &=s\left(\phi_j^i, \pi(v) \right) \quad &\text{for all } v\in \mathbf{V}_0,\\
b(\psi^i_j,q)&=0 & \quad \text{for all } q\in Q_0.
\end{split}
\label{eqn:min_relax} 	
\end{eqnarray}

These basis functions have global support in the domain $\Omeps$, but, as shown in Lemma \ref{lem:exponential-decay}, 
decay exponentially outside some local (oversampled) region. This property plays a vital role in the convergence analysis of the proposed method and justifies the use of local basis functions in $V_{\text{ms}}$. 
Furthermore, we define $V_{\text{glo}} := \text{span} \left \{ \psi_j^i: 1 \leq j \leq \ell_i,  1 \leq i \leq N \right \}$ and $\tilde{V}:=\{\mbf{v}\in \mathbf{V}_0^{\text{div}}: \pi(\mbf{v})=0\}$, where $\mbf{V}_0^{\text{div}}$ is the closed subspace of $\mbf{V}_0$ containing divergence-free vector fields. 
Then, one can show that $\mbf{V}_0^{\text{div}}=V_{\text{glo}} \bigoplus_a  \tilde{V}$. 

\begin{remark}
Suppose that $S \subset \Omega^\epsilon$ is any non-empty connected union of coarse elements $K_i \in \mathcal{T}^H$. Denote $\mathcal{D}_S: \mathbf{H}_0^1(S) \to L_0^2(S)$ the divergence operator corresponding to the set $S$. We have the following auxiliary result from functional analysis. 
\end{remark}

\begin{lemma}[cf. Theorem 6.14-1 in \cite{ciarlet2013linear}] \label{lem:in-sur}
Suppose that $S$ is any non-empty connected union of coarse elements. Restricting the domain of $\mathcal{D}_S$ on the orthogonal complement (with respect to standard $L^2$ inner product) of its kernel, the divergence operator $\mathcal{D}_S$ is injective and surjective. Moreover, it has a continuous inverse and there is a generic constant $\beta_S >0$ such that 
$$ \beta_S \norm{\mathcal{D}_S^{-1} \mu}_a \leq  \norm{\mu} \quad \text{for any } \mu \in L^2_0(S).$$
\end{lemma}
Using the result of Lemma \ref{lem:in-sur}, one can show that \eqref{eqn:min_relax_ms} and \eqref{eqn:min_relax} are well-posed. Let $S$ be the whole domain $\Omega^\epsilon$ or an oversampled region $K_i^+$. Then, for any non-zero element $v \in \mathbf{H}_0^1(S)$, we have 
\begin{eqnarray}
\sup_{v \in \mathbf{H}_0^1(S), v \neq 0} ~ \frac{\abs{b(v, \mu)}}{\norm{v}_a} \geq \frac{\abs{b(\mathcal{D}_S^{-1} \mu, \mu)}}{\norm{\mathcal{D}_S^{-1} \mu}_a} = \frac{\norm{\mu}^2}{\norm{\mathcal{D}_S^{-1} \mu}_a} \geq \beta_S \norm{\mu}
\label{eqn:inf-sup-s}
\end{eqnarray}
for any $\mu \in L_0^2(S)$, which shows that the inf-sup condition holds for \eqref{eqn:min_relax}. Similarly, we can prove the inf-sup condition holds for  \eqref{eqn:min_relax_ms}.

\subsection{The multiscale method}
From the above, we have the multiscale space $V_{\text{ms}}$ for the approximation of velocity field. 
The multiscale solution $\mathbf{u}_{\text{ms}} \in V_{\text{ms}}$ is obtained by solving the following equation: 
\begin{eqnarray}\label{bilinear_velocity}
a(\bold{u_{\text{ms}}}, \bold{v}) = \left \langle \mathbf{f}, \bold{v} \right \rangle \quad \text{for all } \bold{v} \in V_{\text{ms}}.
\end{eqnarray}
To approximate the pressure based on coarse grid, we will construct a specific solution space of finite dimension. 
Let $W(K_i):=\{v\in H^1(K_i):  \int_{K_i} v ~dx=0,~b(\mbf{w},v)=0 \tforall \mbf{w}\in (I-\pi) \mbf{V}_0\}$.
We consider the following spectral problem: Find $(q_j^i, \zeta_j^i)\in W(K_i) \times \mathbb{R}$ such that 
\begin{eqnarray}
\mathcal{A}_i(q_j^i, v) =  \zeta_j^i \mathcal{S}_i(q_j^i, v) \quad \tforall  v \in W(K_i), 
\label{eqn:spectral}
\end{eqnarray}
where $\mathcal{A}_i(\cdot, \cdot)$ and $\mathcal{S}_i(\cdot, \cdot)$ are defined as follows:
\begin{eqnarray}
\mathcal{A}_i(u,v) := \int_{K_i} \nabla u \cdot \nabla v~ dx \quad \text{and} \quad \mathcal{S}_i (u,v) := \int_{K_i} \tilde{\kappa} uv ~ dx
\label{eqn:local_form}
\end{eqnarray}
for any $u, v\in H^1(K_i)$. 
Assume that for each $i \in \{ 1, \cdots, N \}$ the eigenvalues $\zeta^i_j$ are arranged in ascending order such that $0\leq \zeta_1^i \leq \zeta_2^i\leq \cdots$. We then define a finite dimensional solution space $Q_H$ as follows:
$$Q_H:=\text{span}\left \{ q_j^i :~ i=1,\cdots,N, ~ j=1,\cdots,  \ell_i \right \}.$$
 Then, we solve the following variational problem over the domain $\Omega^{\epsilon}$: Find $p_{\text{ms}}\in Q_H$ such that 
\begin{eqnarray}\label{pressure:system}
b(\mbf{v}, p_{\text{ms}} ) =  a(\mbf{u}_{\text{ms}}, \mbf{v})-\left \langle \mbf{f}, \mbf{v} \right \rangle  \quad \tforall \mbf{v} \in V_{\text{aux}}.
\end{eqnarray}
Note that $\text{dim}(Q_H)=\text{dim}(V_{\text{aux}})$. To prove the well-posedness of \eqref{pressure:system}, it  suffices to verify inf-sup condition for the bilinear form $b(\cdot,\cdot)$ over $V_{\text{aux}}$ and $Q_H$.
Recall that the variational formulation \eqref{solution:u} is well-posed and inf-sup condition holds for $b(\cdot,\cdot)$ under spaces $\mathbf{V}_0$ and $Q_0$.  Hence, for any $q\in Q_H$, there exists $\mbf{w}\in \mathbf{V}_0$ such that $b(\mbf{w},q)\geq C\norm{\mbf{w}}_a\norm{q}$ for some constant $C>0$. Choosing $ \mbf{v}:=\pi \mbf{w}$, we have $\mbf{v}\in V_{\text{aux}}$ and 
$$b(\mbf{v},q)=b(\pi\mbf{w},q)=b(\mbf{w},q)\geq C\norm{\mbf{w}}_a\norm{q}\geq C\norm{\mbf{v}}_a\norm{q}.$$
 

Therefore, the problem \eqref{pressure:system} is well-posed. Note that the pressure $p$ solves the following equation:
$$ b(\mbf{v}, p) = a(\mathbf{u}, \mbf{v}) -  \left \langle \mathbf{f}, \mathbf{v} \right \rangle \quad \tforall \mbf{v} \in \mbf{V}_0.$$

Then, we have 
$$ b(\mbf{v}, p - p_{\text{ms}}) = a(\mbf{u} - \mathbf{u}_{\text{ms}}, \mbf{v})  \leq \norm{\mbf{u} - \mbf{u}_{\text{ms}}}_a \norm{\mbf{v}}_a,$$
for all $ \mbf{v} \in V_{\text{aux}}$. 
It implies that 
$$\sup_{\mbf{v} \in V_{\text{aux}}} \frac{b(\mbf{v},p - p_{\text{ms}})}{\norm{\mbf{v}}_a} \leq \norm{\mbf{u} - \mbf{u}_{\text{ms}}}_a.$$
The multiscale solution $p_{\text{ms}}$ serves as an approximation of  the solution $p$ and $\norm{p-p_{\text{ms}}}\lesssim \norm{\mbf{u} - \mbf{u}_{\text{ms}}}_a$. 

\section{Convergence analysis}\label{sec:convergence}
In this section, we analyze the proposed method. We denote $\norm{\cdot}_s := \sqrt{s(\cdot,\cdot)}$ the $s$-norm. In particular, $\norm{\mbf{v}}_s^2=\sum_{i=1}^d \norm{v_i}_s^2$ for any $\mbf{v}=(v_1,\cdots,v_d)^T.$
We also denote $\spt(v)$ the support of a given function or vector field. 
We write $a \lesssim b$ if there exists a generic constant $C>0$ such that $ a \leq Cb$. Define $\Lambda := \displaystyle{\min_{1\leq i\leq N} \lambda_{\ell_i +1}^i}$ and $\Gamma:=\displaystyle{\max_{1\leq i\leq N} \lambda^i_{\ell_i}}$. For a given subregion $S \subset \Omeps$, we define local norms $\norm{\mbf{v}}_{a(S)} := \left ( \int_{S} \abs{\nabla \mbf{v}}^2 ~ dx \right )^{1/2}$ and $\norm{\mbf{v}}_{s(S)} := \left ( \int_{S} \tilde \kappa \abs{\mbf{v}}^2 ~ dx \right )^{1/2}$ for any $\mbf{v} \in \mbf{V}_0$. 

Before estimating the error between global and local multiscale basis functions, we introduce some notions that will be used in the analysis. First, we introduce cutoff function with respect to oversampling region. Given a coarse block $K_i \in \mathcal{T}^H$ and a parameter of oversampling $m \in \mathbb{N}$, we recall that $K_{i,m} \subset \Omega^\epsilon$ is an $m$-layer oversampling region corresponding to $K_i$. 

\begin{definition}
For two positive integers $M$ and $m$ with $M>m\geq 1$, we define cutoff function $\chi_i^{M,m} \in \text{span}\{ \chi_j^{\text{ms}} \}_{j=1}^{N_c}$ such that $0 \leq \chi_i^{M,m} \leq 1$ and 
$$ \chi_i^{M,m} = \left \{ \begin{array}{cl}
1 & \text{in } K_{i,m}, \\
0 & \text{in } \Omega^\epsilon \setminus K_{i,M}. 
\end{array} \right .$$
Note that, we have $K_{i,m} \subset K_{i,M} \subset \Omega^\epsilon$ and $\text{spt}( \chi_i^{M,m})\subset K_{i,M}$. 
\end{definition}

First, we establish the following auxiliary results for later use in the analysis. 
\begin{lemma} \label{lemma:aux-ineq}
Let $\bold{v}\in \mbf{V}_0$ and $k \geq 2$ be an integer. Then, the following inequalities hold:
\begin{enumerate}[label=(\roman*)]
\item $ \norm{\mbf{v}}_a\leq \Gamma^{1/2} \norm{\mbf{v}}_s$ if $\bold{v}\in V_{\text{aux}}$;
\item $\norm{\mbf{v}}_s\leq \Lambda^{-1/2}\norm{\mbf{v}}_a$ if $\bold{v}\notin V_{\text{aux}}$;
\item $\norm{\mbf{v}}_s^2\lesssim  \Lambda^{-1}\norm{ (I-\pi)\mbf{v}}_a^2+\norm{\pi\mbf{v}}_s^2$;
\item $\norm{(1-\chi^{k,k-1}_i)\mbf{v} }_a^2\leq 2(1+\Lambda^{-1}) \norm{\mbf{v}}^2_{a(\Omega^\epsilon\setminus K_{i,k-1} )}+2\norm{\pi\mbf{v}}^2_{s(\Omega^\epsilon\setminus K_{i,k-1} )}$
\item $\norm{(1-\chi^{k,k-1}_i)\mbf{v} }_s^2\leq \Lambda^{-1} \norm{\mbf{v}}^2_{a(\Omega^\epsilon\setminus K_{i,k-1}) }+\norm{\pi\mbf{v}}^2_{s(\Omega^\epsilon\setminus K_{i,k-1}) }.$
\end{enumerate}
\end{lemma}

\begin{proof}
Note that one can write $\bold{v}=\sum_{i=1}^N\sum_{j\geq 1} \alpha^i_j \phi^i_j$ with $\alpha_j^i \in \mathbb{R}$ for any $\bold{v}\in \mbf{V}_0$. 
\begin{enumerate}[label=(\roman*)]
\item Since $\mbf{v}\in V_{\text{aux}}$, then $\alpha^i_j=0$ for $j\geq \ell_i+1$. Using the local spectral problem \eqref{eqn:spectral}, we obtain
\[
\norm{\mbf{v}}_a^2=\sum_{i=1}^N\sum_{j= 1}^{\ell_i} \alpha^i_j a(\phi^i_j, \mbf{v})=\sum_{i=1}^N\sum_{j= 1}^{\ell_i} \alpha^i_j \lambda^i_j  s(\phi^i_j, \mbf{v})
\leq \Gamma \sum_{i=1}^N\sum_{j= 1}^{\ell_i} \alpha^i_j   s(\phi^i_j, \mbf{v})=\Gamma\norm{\mbf{v}}^2_s.
\]
\item For any $\bold{v}\notin V_{\text{aux}}$, one can write $\bold{v}=\sum_{i=1}^N\sum_{j\geq \ell_i+1} \alpha^i_j \phi^i_j$. Then, we have 
\[ 
\norm{\mbf{v}}_a^2=\sum_{i=1}^N\sum_{j\geq \ell_i+1} \alpha^i_j a(\phi^i_j, \mbf{v})=\sum_{i=1}^N\sum_{j\geq \ell_i+1} \alpha^i_j \lambda^i_j  s(\phi^i_j, \mbf{v})
\geq \Lambda \sum_{i=1}^N\sum_{j\geq \ell_i+1} \alpha^i_j   s(\phi^i_j, \mbf{v})=\Lambda\norm{\mbf{v}}^2_s.
\]
\item The result follows from (i), (ii), and the triangle inequality.
\item By using the property of cutoff function $\chi^{k,k-1}_i$ and (iii), we have
\begin{eqnarray*}
\begin{split}
\norm{(1-\chi^{k,k-1}_i)\mbf{v} }_a^2&\leq 2 \left( \int_{\Omega^\epsilon\setminus K_{i,k-1} } (1-\chi_i^{k,k-1})^2| \nabla \mbf{v}|^2+|\mbf{v}\nabla \chi_i^{k,k-1}|^2  dx\right)\\
&\leq 2\left (\norm{\mbf{v}}^2_{a(\Omega^\epsilon\setminus K_{i,k-1} )}+\norm{\mbf{v}}^2_{s(\Omega^\epsilon\setminus K_{i,k-1} )} \right) \\
&\leq 2(1+\Lambda^{-1}) \norm{\mbf{v}}^2_{a(\Omega^\epsilon\setminus K_{i,k-1} )}+2\norm{\pi\mbf{v}}^2_{s(\Omega^\epsilon\setminus K_{i,k-1} )}.
\end{split}
\end{eqnarray*}
\item  For any $k\geq 2$, we have
\begin{eqnarray*}
\begin{split}
\norm{(1-\chi^{k,k-1}_i)\mbf{v} }_s^2 &\leq \norm{\mbf{v}}^2_{s(\Omega^\epsilon\setminus K_{i,k-1}) }\\
&\leq \Lambda^{-1} \norm{\mbf{v}}^2_{a(\Omega^\epsilon\setminus K_{i,k-1}) }+\norm{\pi\mbf{v}}^2_{s(\Omega^\epsilon\setminus K_{i,k-1}) }.
\end{split}
\end{eqnarray*}
\end{enumerate}
This completes the proof. 
\end{proof}

First, we present the convergence of using global basis functions constructed in \eqref{eqn:min_relax}. We define $\bold{u_{\text{glo}}}\in V_{\text{glo}}$ as the global multiscale solution satisfying 
 \begin{eqnarray}\label{solution:u_glo}
 a(\bold{u_{\text{glo}}},\bold{v})= \left \langle \bold{f},\bold{v} \right \rangle \quad \text{for all } \mathbf{v}\in V_{\text{glo}}.
 \end{eqnarray}
 
 \begin{theorem}
 Let $\bold{u}$ be the solution of \eqref{solution:u} and $\bold{u_{\text{glo}}}$ be the solution of \eqref{solution:u_glo}. We have
$$
 \norm{\bold{u}-\bold{u_{\text{glo}}}}_a\lesssim  \Lambda^{-1/2} \norm{\tilde{\kappa}^{-1/2} \bold{f} }.
$$
Moreover, if $\{ \chi_j^{\text{ms}} \}_{j=1}^{N_c}$ is a set of bilinear partition of unity, we have 
$$ \norm{\mbf{u} - \mbf{u}_{\text{glo}}}_a \lesssim H \Lambda^{-1} \norm{\mbf{f} }.$$
 \end{theorem}

\begin{proof}
By the definition of $\bold{u}$ and $\bold{u_{\text{glo}}}$, we have 
\begin{eqnarray}
 a(\bold{u}-\bold{u_{\text{glo}}}, \bold{v})=0 \quad \tforall ~\bold{v}\in V_{\text{glo}}.
\end{eqnarray}
Hence, we have $\bold{u}-\bold{u_{\text{glo}}} \in \tilde V$ and 
 \begin{eqnarray}\label{thm:ineqn1}
 \begin{aligned}
 a(\bold{u}-\bold{u_{\text{glo}}}, \bold{u}-\bold{u_{\text{glo}}})&=a(\bold{u}-\bold{u_{\text{glo}}}, \bold{u})=(\bold{f}, \bold{u}-\bold{u_{\text{glo}}}) \leq \norm{\tilde{\kappa}^{-\frac{1}{2}} \bold{f} } \norm{\mbf{u}-\mbf{u}_{\text{glo}}}_s.
 \end{aligned} 
 \end{eqnarray}
 Since $\mbf{u}-\mbf{u}_{\text{glo}}\in \mbf{V}_0-V_{\text{aux}}$, it follows from Lemma \eqref{lemma:aux-ineq} (ii) that 
\begin{eqnarray}\label{thm:ineqn2}
\norm{\mathbf{u}-\mathbf{u}_{\text{glo}}}^2_s\leq \Lambda^{-1}  \norm{\mathbf{u}-\mathbf{u}_{\text{glo}}}^2_a.
\end{eqnarray}
The result follows by combining \eqref{thm:ineqn1} and \eqref{thm:ineqn2}. The second part follows from the fact that $\abs{\nabla\chi_j^{\text{ms}}} = O(H^{-1})$ when $\{ \chi_j^{\text{ms}} \}_{j=1}^{N_c}$ is a set of bilinear partition of unity functions. 
\end{proof}

Next we analyze the convergence of the proposed multiscale method. We first recall Projection Theorem, which can be found in many functional analysis literature, e.g., \cite[Section 4.3]{ciarlet2013linear}.
\begin{theorem}[Projection Theorem]
Let $\mathcal{V}$ be a closed subspace of the Hilbert space $\mathcal{H}$ equipped with an inner product $(\cdot,\cdot)_{\mathcal{H}}$. Then, for any given element $f \in \mathcal{H}$, there exists a unique element $p \in \mathcal{V}$ such that 
$$ \norm{f- p} = \min_{v \in \mathcal{V}} \norm{f-v}.$$
Here, $\norm{\cdot}$ is the norm induced by the inner product $(\cdot,\cdot)_{\mathcal{H}}$. Moreover, the mapping $\mathbb{P} : f \mapsto p$ is linear and satisfies the inequality $\norm{\mathbb{P}f} \leq \norm{f}$ for any $f \in \mathcal{H}$. 
\end{theorem}

In the following lemma, we show the existence of a projection from $\mathbf{V}_0(D) $ to $\mathbf{V}_0^{\text{div}}(D)$ using the Projection Theorem. 
\begin{lemma}\label{lemma:projection}
Let $\mathcal{D}\subseteq \Omeps$. Then, there exists a divergence-free projection $\mathbb{P}_\mathcal{D}: \mathbf{V}_0(\mathcal{D})\to \mathbf{V}_0^{\text{div}}(\mathcal{D})$, where $\mathbf{V}_0^{\text{div}}(\mathcal{D}):=\{ \mbf{v}\in\mathbf{V}_0(\mathcal{D}): b(\mbf{v},q)=0 \tforall~ q\in L^2(\mathcal{D}) \} $.
\end{lemma}
\begin{proof}
Define a bilinear form on $\mathbf{V}_0(\mathcal{D})$ as follows: $(\mbf{u},\mbf{v})_{as(\mathcal{D})}:=a_\mathcal{D}(\mbf{u},\mbf{v})+s_\mathcal{D}(\mbf{u},\mbf{v})$, where $a_\mathcal{D}(\cdot,\cdot)$ and $s_\mathcal{D}(\cdot,\cdot)$ are the restriction of $a(\cdot,\cdot)$ and $s(\cdot,\cdot)$ on the subregion $\mathcal{D}$. One can easily show that $(\cdot,\cdot)_{as}$ is an inner product defined on $\mbf{V}_0(\mathcal{D})$. 

Next, we show that $\mathbf{V}_0^{\text{div}}(\mathcal{D}) $ is a closed subspace of $\mathbf{V}_0(\mathcal{D}) $ with respect to the inner product $(\cdot, \cdot)_{as}$. Let $\{\mbf{f}_n\}$ be a sequence in $\mathbf{V}_0^{\text{div}}(\mathcal{D}) $ that converges to $\mbf{f}$ in $\mathbf{V}_0(\mathcal{D}) $. Since $\mbf{f}_n\in \mathbf{V}_0^{\text{div}}(\mathcal{D})$, then we have $b(\mbf{f}_n,g)=0 \tforall g\in L^2(\mathcal{D})$. Then $\displaystyle{\lim_{n\to \infty} b(\mbf{f}_n,g)=b(\mbf{f},g)=0} \tforall g\in L^2(\mathcal{D})$. It implies that $\mbf{f}\in \mathbf{V}_0^{\text{div}}(\mathcal{D})$. Consequently, $\mathbf{V}_0^{\text{div}}(\mathcal{D})$ is a closed subspace of $\mathbf{V}_0(\mathcal{D}) $. An application of Projection Theorem proves the desired result.
\end{proof}

\begin{remark}
We denote $\norm{\cdot}_{as(\mathcal{D})}$ the norm induced by the inner product $(\cdot,\cdot)_{as(\mathcal{D})}$. Then, we have $\norm{\mathbb{P}_\mathcal{D}(\mbf{v}) }_{as(\mathcal{D})}\leq \norm{\mbf{v}}_{as(\mathcal{D})}$ for any $v \in \mbf{V}_0(\mathcal{D})$. We simply write $\norm{\cdot}_{as}$ in short for $\norm{\cdot}_{as(\mathcal{D})}$ when $\mathcal{D} = \Omeps$. 
Moreover, the subscript $\mathcal{D}$ will be dropped from $\mathbb{P}_\mathcal{D}$ when there is no ambiguity. 
\end{remark}

\begin{lemma}\label{lemma:properties}
For any auxiliary function $v_{\text{aux}} \in V_{\text{aux}}$, there exists a function $z \in \mathbf{V}_0^{\text{div}}$ such that 
$$ \pi (z) = v_{\text{aux}}, \quad \norm{z}_a^2 \leq D \norm{v_{\text{aux}}}_s^2, \tand \spt(z) \subseteq \spt (v_{\text{aux}}).$$
Here, $D$ is a generic constant depending only on the coarse mesh, the partition of unity, and the eigenvalues obtained in \eqref{eqn:spectral}.
\end{lemma} 
\begin{proof}
Without loss of generality, we can assume that $v_{\text{aux}} \in V_{\text{aux}}^i$. Consider the following variational problem: Find $z \in \mbf{V}_0^{\text{div}}(K_i)$ 
and $\mu \in V_{\text{aux}}^i$ such that 
\begin{eqnarray}
\begin{split}
a_i(z,v) + s_i(v,\mu) & = 0 &\quad \tforall v \in \mbf{V}_0^{\text{div}}(K_i), \\
s_i (z, q) & = s_i(v_{\text{aux}},q) &\quad \tforall q \in V_{\text{aux}}^i.
\end{split}
\label{eqn:local-var}
\end{eqnarray}
Here, the bilinear forms $a_i(\cdot,\cdot)$ and $s_i(\cdot,\cdot)$ are defined in \eqref{eqn:local_form}. We will show the well-posedness of the problem \eqref{eqn:local-var}. It suffices to show that there is a function $z \in \mbf{V}_0^{\text{div}}(K_i)$ such that 
$$ s_i(z, v_{\text{aux}} ) \geq C_1 \norm{v_{\text{aux}}}_{s(K_i)}^2 \tand \norm{z}_{a(K_i)}^2 \leq C_2 \norm{v_{\text{aux}}}_{s(K_i)}^2$$
for some generic constants $C_1$ and $C_2$. We denote $\mathcal{I}_{K_i} := \{ j: x_j \text{ is a coarse vertex of } K_i \}$ and define $B := \prod_{j \in \mathcal{I}_{K_i}} \chi_j^{\text{ms}}$. Taking  $z =\mbp( Bv_{\text{aux}})$, we have 
$$s_i(z,v_{\text{aux}}) =  s_i(\mathbb{P}( Bv_{\text{aux}}), v_{\text{aux}}) = \int_{K_i} \tilde \kappa \mbp(Bv_{\text{aux}})v_{\text{aux}} ~ dx \geq C_{\pi}^{-1} \norm{v_{\text{aux}}}_{s(K_i)}^2. $$
Here, the constant $C_{\pi}$ is defined to be 
$$ C_{\pi} := \sup_{K \in \mathcal{T}^H,~ \mu \in V_{\text{aux}}} \frac{ \int_K \tilde \kappa \mu^2 ~ dx}{\int_K \tilde \kappa  \mathbb{P}( B\mu) \mu ~ dx}>0.$$
Note that $\abs{B} \leq 1$ and $\abs{\nabla B}^2 \leq C_\mathcal{T} \sum_{j \in \mathcal{I}_{K_i}} \abs{\chi_j^{\text{ms}}}^2$ with $C_{\mathcal{T}} := \displaystyle{\max_{K \in \mathcal{T}^H} \abs{\mathcal{I}_K}^2}$. The following inequalities hold
$$\norm{z}_{a(K_i)} \leq \norm{\mbp( Bv_{\text{aux}}) }_{as(K_i)} \leq \norm{ Bv_{\text{aux}}}_{as(K_i)}\lesssim \norm{ Bv_{\text{aux}}}_{a(K_i)}.$$
Then, we have 
 $$\norm{z}_{a(K_i)}^2 \lesssim \norm{B v_{\text{aux}}}_{a(K_i)}^2\leq C_{\mathcal{T}} C_{\pi} (1+ \Gamma) \norm{v_{\text{aux}}}_{s(K_i)}^2.$$
It shows the existence and uniqueness of the function $z$ for a given auxiliary function $v_{\text{aux}} \in V_{\text{aux}}^i$. From the second equality in \eqref{eqn:local-var}, we see that $\pi_i(z) = v_{\text{aux}}$. This completes the proof.  
\end{proof}

The following lemma shows that the global multiscale basis functions have a decay property. 
 \begin{lemma} \label{lem:exponential-decay}
Let $\phi_j^i \in V_{\text{aux}}$ be a given auxiliary function. Suppose that $\psi_{j,\text{ms}}^i$ is a multiscale basis function obtained in \eqref{eqn:min_relax_ms} over the oversampling domain $K_{i,k}$ with $k \geq 2$ and $\psi_j^i$ is the corresponding global basis function obtained in \eqref{eqn:min_relax}. Then, the following estimate holds: 
$$ 
\norm{\psi^i_j-\psi^i_{j,\text{ms}} }_a^2  + \norm{\pi(\psi^i_j-\psi^i_{j,\text{ms}})}_s^2 \leq 
E\left ( \norm{\psi_j^i}_a^2 + \norm{\pi(\psi_j^i)}_s^2 \right ),
$$
where $E = 3\left ( 1+\Lambda^{-1} \right ) \left ( 1+ \left [6 (1+\Lambda^{-1})\right]^{-1/2}  \right )^{1-k}$ is a factor of exponential decay. 
\end{lemma}
\begin{proof}
Subtracting the first equation of \eqref{eqn:min_relax_ms} from that of \eqref{eqn:min_relax}, we obtain 
$$ a(\psi^i_j-\psi^i_{j,\text{ms}}, v) + s(\pi(\psi^i_j-\psi^i_{j,\text{ms}}), \pi(v)) + b(v, \xi_j^i - \xi_{j,\text{ms}}^i) = 0 \quad \tforall v \in \mathbf{V}_0(K_{i,k}).$$
Taking $v = w - \psi_{j,\text{ms}}^i$ with $w\in \mathbf{V}_0^{\text{div}}(K_{i,k})$, 
then we have
\begin{equation}
a(\psi^i_j-\psi^i_{j,\text{ms}}, \psi_{j,\text{ms}}^i ) + s(\pi(\psi^i_j-\psi^i_{j,\text{ms}}), \pi(\psi_{j,\text{ms}}^i)) = a(\psi^i_j-\psi^i_{j,\text{ms}}, w ) + s(\pi(\psi^i_j-\psi^i_{j,\text{ms}}), \pi(w)).\label{eqn:decay0}
\end{equation}
Utilizing \eqref{eqn:decay0} and Cauchy-Schwarz inequality, one can show that 
$$\norm{\psi^i_j-\psi^i_{j,\text{ms}} }_a^2 + \norm{\pi(\psi^i_j-\psi^i_{j,\text{ms}})}_s^2 \leq  \norm{\psi_j^i - w}_a^2 + \norm{\pi(\psi_j^i - w)}_s^2$$
for any $w \in \mathbf{V}^{\text{div}}_0(K_{i,k})$. 
Let $w = \mathbb{P}(\chi_i^{k,k-1} \psi_j^i)$. Note that $\psi_j^i = \mathbb{P}(\psi_j^i)$. 
Then, we have 
\begin{eqnarray}
\begin{split}
\norm{\psi^i_j-\psi^i_{j,\text{ms}} }_a^2 + \norm{\pi(\psi^i_j-\psi^i_{j,\text{ms}})}_s^2 & \leq  \norm{\psi_j^i - \mathbb{P}(\chi_i^{k,k-1} \psi_j^i)}_a^2 + \norm{\pi(\psi_j^i -\mathbb{P}(\chi_i^{k,k-1} \psi_j^i))}_s^2 \\
& 
\leq \norm{(1-\chi_i^{k,k-1})\psi_j^i}_a^2 + \norm{(1- \chi_i^{k,k-1})\psi_j^i}_s^2.
\end{split}
\label{eqn:decay-1}
\end{eqnarray}

Using (iv) and (v) of Lemma \ref{lemma:aux-ineq}, we have 
\begin{eqnarray}
\norm{\psi^i_j-\psi^i_{j,\text{ms}} }_a^2 + \norm{\pi(\psi^i_j-\psi^i_{j,\text{ms}})}_s^2 \leq 3 (1+ \Lambda^{-1})\left ( \norm{\psi_j^i}_{a(\Omeps \setminus K_{i,k-1})}^2 + \norm{\pi(\psi_j^i)}_{s(\Omeps \setminus K_{i,k-1})}^2 \right ).
\label{eqn:decay-3}
\end{eqnarray}
Next, we estimate the term $\norm{\psi_j^i}_{a(\Omeps \setminus K_{i,k-1})}^2 + \norm{\pi(\psi_j^i)}_{s(\Omeps \setminus K_{i,k-1})}^2$. We claim that it can be bounded by the term $F^2 := \norm{\psi_j^i}_{a(K_{i,k-1} \setminus K_{i,k-2} )}^2 + \norm{\pi(\psi_j^i)}_{s(K_{i,k-1} \setminus K_{i,k-2} )}^2$. This recursive property is crucial in our convergence estimate.

Note that $\text{spt}( 1- \chi_i^{k-1, k-2}) \subseteq \Omeps \setminus K_{i,k-2}$ and $\text{spt} ( \phi_j^i ) \subseteq K_i$. So $s(\phi^i_j, \pi \mathbb{P}( (1- \chi_i^{k-1,k-2}) \psi_j^i))=0$. Choosing test function $v = \mathbb{P}( (1- \chi_i^{k-1,k-2}) \psi_j^i)$ in the variational formulation \eqref{eqn:min_relax}, we have
\begin{eqnarray}
 \quad a\left (\psi_j^i , \mathbb{P}( (1- \chi_i^{k-1,k-2}) \psi_j^i) \right ) + s \left ( \pi(\psi_j^i), \pi(\mathbb{P}( (1- \chi_i^{k-1,k-2}) \psi_j^i)) \right ) = 0.
\label{eqn:decay-3.1}
\end{eqnarray}
Note that 
\begin{eqnarray*}
\begin{split}
a\left (\psi_j^i , \mathbb{P}( (1- \chi_i^{k-1,k-2}) \psi_j^i) \right )  & = \int_{\Omeps \setminus K_{i,k-2}} \nabla \psi_j^i : \nabla \left ( \mathbb{P}( (1- \chi_i^{k-1,k-2}) \psi_j^i) \right ) ~ dx \\
& = \int_{\Omeps \setminus K_{i,k-2}} \abs{\nabla \psi_j^i}^2 ~ dx - \int_{\Omeps \setminus K_{i,k-2}} \nabla \psi_j^i : \nabla \left (\mathbb{P}(\chi_i^{k-1,k-2} \psi_j^i) \right )~ dx.
\end{split}
\end{eqnarray*}
Consequently, we have 
\begin{eqnarray}
\begin{split}
\norm{\psi_j^i}_{a(\Omeps \setminus K_{i,k-1})}^2 & \leq \int_{\Omeps \setminus K_{i,k-2}} \abs{\nabla \psi_j^i}^2 ~ dx \\
& = a\left (\psi_j^i , \mathbb{P}( (1- \chi_i^{k-1,k-2}) \psi_j^i) \right ) + \int_{\Omeps \setminus K_{i,k-2}} \nabla \psi_j^i : \nabla \left (\mathbb{P}(\chi_i^{k-1,k-2} \psi_j^i) \right )~ dx \\
& \leq a\left (\psi_j^i , \mathbb{P}( (1- \chi_i^{k-1,k-2}) \psi_j^i) \right ) + \norm{\psi_j^i}_{a(K_{i,k-1} \setminus K_{i,k-2})} \norm{\mathbb{P}(\chi_i^{k-1,k-2} \psi_j^i)}_{as(K_{i,k-1} \setminus K_{i,k-2})}.
\end{split}
\label{eqn:decay-4}
\end{eqnarray}

Note that $\chi_i^{k-1,k-2} \equiv 0$ in $\Omeps \setminus K_{i,k-1}$. Thus, we have 
\begin{eqnarray*}
\begin{split}
s \left ( \pi(\psi_j^i), \pi(\mathbb{P}( (1- \chi_i^{k-1,k-2}) \psi_j^i)) \right ) & = 
\norm{\pi(\psi_j^i)}_{s(\Omeps \setminus K_{i,k-1})}^2 \\
& ~ + \int_{K_{i,k-1} \setminus K_{i,k-2}} \tilde \kappa \pi(\psi_j^i) \pi \left (\mathbb{P}( (1-\chi_i^{k-1,k-2}) \psi_j^i) \right )~ dx
\end{split}
\end{eqnarray*}
and 
\begin{eqnarray}
\begin{split}
& \quad \norm{\pi(\psi_j^i)}_{s(\Omeps \setminus K_{i,k-1})}^2  \\
& = s \left ( \pi(\psi_j^i), \pi(\mathbb{P}( (1- \chi_i^{k-1,k-2}) \psi_j^i)) \right ) - \int_{K_{i,k-1} \setminus K_{i,k-2}} \tilde \kappa \pi(\psi_j^i) \pi \left (\mathbb{P}( (1-\chi_i^{k-1,k-2}) \psi_j^i) \right )~ dx \\
& \leq s \left ( \pi(\psi_j^i), \pi(\mathbb{P}( (1- \chi_i^{k-1,k-2}) \psi_j^i)) \right )  + \norm{\pi(\psi_j^i)}_{s(K_{i,k-1} \setminus K_{i,k-2})} \norm{\mathbb{P}( (1-\chi_i^{k-1,k-2}) \psi_j^i) }_{as(K_{i,k-1} \setminus K_{i,k-2})}  .
\end{split}
\label{eqn:decay-5}
\end{eqnarray}

Using (iv) and (v) of Lemma \ref{lemma:aux-ineq}, one can show that 
\begin{eqnarray}
\begin{split}
&\norm{\mathbb{P}( (1-\chi_i^{k-1,k-2}) \psi_j^i) }^2_{as(K_{i,k-1} \setminus K_{i,k-2})}\leq  3(1+\Lambda^{-1})F^2,\\ 
&\norm{\mathbb{P}(\chi_i^{k-1,k-2} \psi_j^i)}^2_{as(K_{i,k-1} \setminus K_{i,k-2})} \leq 3(1+\Lambda^{-1})F^2.
\end{split}\label{ineq:decay-67}
\end{eqnarray}

Combining \eqref{eqn:decay-3.1} and the inequalities \eqref{eqn:decay-4} -- \eqref{ineq:decay-67}, we have 
\begin{eqnarray}
\begin{split}
& \quad \norm{\psi_j^i}_{a(\Omeps \setminus K_{i,k-1})}^2 + \norm{\pi(\psi_j^i)}_{s(\Omeps \setminus K_{i,k-1})}^2 \\
& \leq  \left ( \norm{\psi_j^i}_{a(K_{i,k-1} \setminus K_{i,k-2})}^2 + \norm{\pi(\psi_j^i)}_{s(K_{i,k-1} \setminus K_{i,k-2})}^2 \right )^{1/2} \left [6 (1+\Lambda^{-1})\right]^{1/2}F\\
& = \left [6 (1+\Lambda^{-1})\right]^{1/2}F^2.
\end{split}
\label{eqn:decay-6}
\end{eqnarray}
Notice that, using the inequality \eqref{eqn:decay-6}, we have
\begin{eqnarray*}
\begin{split}
& \quad \norm{\psi_j^i}_{a(\Omeps \setminus K_{i,k-2})}^2 + \norm{\pi(\psi_j^i)}_{s(\Omeps \setminus K_{i,k-2})}^2 \\
& = \norm{\psi_j^i}_{a(\Omeps \setminus K_{i,k-1})}^2 + \norm{\pi(\psi_j^i)}_{s(\Omeps \setminus K_{i,k-1})}^2 + \norm{\psi_j^i}_{a(K_{i,k-1} \setminus K_{i,k-2})}^2 + \norm{\pi(\psi_j^i)}_{s(K_{i,k-1} \setminus K_{i,k-2})}^2 \\
& \geq 
\left ( 1+ \left [6 (1+\Lambda^{-1})\right]^{-1/2} \right ) \left ( \norm{\psi_j^i}_{a(\Omeps \setminus K_{i,k-1})}^2 + \norm{\pi(\psi_j^i)}_{s(\Omeps \setminus K_{i,k-1})}^2 \right ).
\end{split}
\end{eqnarray*}
Using the above inequality recursively, we obtain 
$$ \norm{\psi_j^i}_{a(\Omeps \setminus K_{i,k-1})}^2 + \norm{\pi(\psi_j^i)}_{s(\Omeps \setminus K_{i,k-1})}^2 \leq  \left ( 1+ \left [6 (1+\Lambda^{-1})\right]^{-1/2} \right )^{1-k} \left ( \norm{\psi_j^i}_a^2 + \norm{\pi(\psi_j^i)}_s^2 \right ).$$
This completes the proof. 
\end{proof}

The above lemma shows that the global multiscale basis is localizable. We need the following result to show the convergence estimate.  

\begin{lemma} \label{lem:sum-split}
With the same notations in Lemma \ref{lem:exponential-decay}, we have 
$$ 
\norm{\displaystyle\sum_{i=1}^N (\psi^i_j-\psi^i_{j,\text{ms}}) }^2_a +\norm{\displaystyle\sum_{i=1}^N \pi(\psi^i_j - \psi^i_{j,\text{ms}}) }^2_s\lesssim (k+1)^d  \sum_{i=1}^N \left [ \norm{\psi^i_j-\psi^i_{j,\text{ms}}}_a^2 + \norm{ \pi(\psi^i_j-\psi^i_{j,\text{ms}} )}_s^2 \right ].
$$
 \end{lemma}

\begin{proof}
Denote $w:=\displaystyle{\sum_{i=1}^N} (\psi^i_j-\psi^i_{j,\text{ms}}) $. 
Noice that, for any $i\in \{1,\cdots, N\}$, it holds that 
\begin{equation}\label{variable system:sum}
\begin{aligned}
a(\psi^i_j-\psi^i_{j,\text{ms}},v)+s(\pi(\psi^i_j-\psi^i_{j,\text{ms}}),\pi(v))&=0, \quad \tforall v\in \mathbf{V}^{\text{div}}_0(K_{i,k}).
\end{aligned}
\end{equation}

Choosing  $v=\mathbb{P}((1-\chi_i^{k+1,k})w)$ in \eqref{variable system:sum}, 
we have 
$$
a\left (\psi^i_j-\psi^i_{j,\text{ms}},\mathbb{P}((1-\chi_i^{k+1,k})w)\right)+s\left (\pi(\psi^i_j-\psi^i_{j,\text{ms}}),\pi (\mathbb{P}((1-\chi_i^{k+1,k})w))\right ) = 0.
$$
Note that $\mathbb{P}(w)=w$. Hence, we have 
\begin{align*}
\norm{w}_a^2+\norm{\pi w}_s^2&=\sum_{i=1}^N  a(\psi^i_j-\psi^i_{j,\text{ms}},w )+  s(\pi(\psi^i_j-\psi^i_{j,\text{ms}}),\pi(w) )  \\
&=\sum_{i=1}^N  a\left (\psi^i_j-\psi^i_{j,\text{ms}},\mathbb{P}(\chi_i^{k+1,k} w) \right )+  s\left (\pi(\psi^i_j-\psi^i_{j,\text{ms}}),\pi( \mathbb{P} (\chi_i^{k+1,k} w)) \right ).
\end{align*}
For each $i \in \{ 1,\cdots, N\}$, using the properties of the cutoff function $\chi_i^{k+1,k}$ and (ii) of Lemma \ref{lemma:aux-ineq}, we have the following estimates: 
\begin{eqnarray}
\begin{split}
\norm{\chi_i^{k+1,k} w}_a^2&
\lesssim \norm{w}^2_{s(K_{i,k+1})}+\norm{w}^2_{a(K_{i,k+1})}
\leq (1+\Lambda^{-1})  \left (\norm{w}^2_{a(K_{i,k+1})}+\norm{\pi(w)}^2_{s(K_{i,k+1})}\right ),\\
\norm{\pi(\chi_i^{k+1,k} w )}_{s}^2 &\leq  \norm{ \chi_i^{k+1,k} w}_{s(K_{i,k+1})}^2 \leq \Lambda^{-1}\norm{w}_{a(K_{i,k+1})}^2 +  \norm{\pi(w)}_{s(K_{i,k+1})}^2. 
\end{split}
\label{ineq:1}
\end{eqnarray}
Furthermore, an application of \eqref{ineq:1} we arrive at the following estimate:
\begin{eqnarray}
\begin{split}
\norm{\chi_i^{k+1,k} w}_{as}^2&=  \norm{\chi_i^{k+1,k} w}_a^2+\norm{\pi(\chi_i^{k+1,k} w)}_s^2 \lesssim  \norm{w}^2_{a(K_{i,k+1})}+\norm{\pi(w)}^2_{s(K_{i,k+1})}
\end{split}
\label{ineq:2}
\end{eqnarray}

Combining \eqref{ineq:1} and \eqref{ineq:2}, we have
\begin{align*}
\norm{w}_a^2+\norm{\pi(w)}_s^2&\leq \sum_{i=1}^N  \norm{ \psi^i_j-\psi^i_{j,\text{ms}}}_a \cdot \norm{\chi_i^{k+1,k} w  } _{as} +\norm{ \pi(\psi^i_j-\psi^i_{j,\text{ms}} )}_s \cdot  \norm{\chi_i^{k+1,k} w }_{as} \\
&\lesssim \sum_{i=1}^N \left (  \norm{ \psi^i_j-\psi^i_{j,\text{ms}}}_a^2 + \norm{ \pi(\psi^i_j-\psi^i_{j,\text{ms}} )}_s^2 \right)^{1/2} \cdot  \left ( \norm{w}^2_{a(K_{i,k+1})}+\norm{\pi(w)}^2_{s(K_{i,k+1})} \right )^{1/2} \\
& \lesssim \left ( \sum_{i=1}^N \norm{ \psi^i_j-\psi^i_{j,\text{ms}}}_a^2 + \norm{ \pi(\psi^i_j-\psi^i_{j,\text{ms}} )}_s^2 \right )^{1/2} \left ( \sum_{i=1}^N \norm{w}_{a(K_{i,k+1})}^2 + \norm{ \pi(w) }_{s(K_{i,k+1})}^2 \right)^{1/2}\\
& \lesssim (k+1)^{d/2} \left ( \sum_{i=1}^N \norm{ \psi^i_j-\psi^i_{j,\text{ms}}}_a^2 + \norm{ \pi(\psi^i_j-\psi^i_{j,\text{ms}} )}_s^2 \right )^{1/2} \left ( \norm{w}_a^2+\norm{\pi(w)}_s^2 \right )^{1/2}.
\end{align*}
Therefore, we have 
$$ \norm{w}_a^2+\norm{\pi(w)}_s^2 \lesssim (k+1)^d  \sum_{i=1}^N \left [ \norm{\psi^i_j-\psi^i_{j,\text{ms}}}_a^2 + \norm{ \pi(\psi^i_j-\psi^i_{j,\text{ms}} )}_s^2 \right ].$$

This completes the proof. 
\end{proof} 
 

Finally, we state and prove the main result of this work. It reads as follows. 
\begin{theorem}
Let $\mathbf{u}$ be the solution of \eqref{solution:u} and $\mathbf{u}_{\text{ms}}$ be the solution of \eqref{bilinear_velocity}. Then, we have 
$$
\norm{\mathbf{u}-\mathbf{u}_{\text{ms}}}_a \lesssim \Lambda^{-1} \norm{\tilde \kappa^{-1/2} \mathbf{f}} + \max \{ \tilde \kappa \} (k+1)^{d/2} E^{1/2} (1+D)\norm{\mathbf{u}_{\text{glo}}}_s,
$$
where $\mathbf{u}_{\text{glo}}$ is the solution of \eqref{solution:u_glo}. Moreover, if the oversampling parameter $k$ is sufficiently large and $\{ \chi_i^{\text{ms}} \}_{i=1}^{N_c}$ is a set of bilinear partition of unity, we have 
$$ \norm{\mathbf{u}-\mathbf{u}_{\text{ms}}}_a \lesssim H \Lambda^{-1} \norm{\mathbf{f}}.$$
\end{theorem}
\begin{proof}
It follows from Galerkin orthogonality that $ \norm{\mathbf{u}-\mathbf{u}_{\text{ms}}}_a\leq \norm{\mathbf{u}-\mathbf{v}}_a$ for any $\mathbf{v}\in V_{\text{ms}}$.
We write $\displaystyle{\mbf{u}_{\text{glo}} := \sum_{i=1}^N  \sum_{j=1}^{\ell_i} c_{ij} \psi_j^i}$
and define a function $\mbf{v}$ such that 
$\displaystyle{ \mbf{v} := \sum_{i=1}^N  \sum_{j=1}^{\ell_i} c_{ij} \psi_{j,\text{ms}}^i.}$
Then, we have 
\begin{eqnarray*}
\begin{split}
\norm{\mbf{u} - \mbf{u}_{\text{ms}}}_a & \leq \norm{\mbf{u} - \mbf{v}}_a \leq \norm{\mbf{u} - \mbf{u}_{\text{glo}}}_a + \norm{\mbf{u}_{\text{glo}} - \mbf{v}}_a.
\end{split}
\end{eqnarray*}
The first term of the right-hand side can be estimated by the result of \eqref{thm:ineqn2}. It suffices to estimate the second term. By Lemmas \ref{lem:exponential-decay} and \ref{lem:sum-split}, we have
\begin{eqnarray*}
\begin{split}
\norm{\mbf{u}_{\text{glo}} - \mbf{v}}_a^2 & = \norm{ \sum_{i=1}^N \sum_{j=1}^{\ell_i} c_{ij} (\psi_j^i - \psi_{j,\text{ms}}^i)}_a^2\\
& \leq C(k+1)^d \sum_{i=1}^N\left( \norm{\sum_{j=1}^{\ell_i} c_{ij} (\psi_j^i - \psi_{j,\text{ms}}^i)}_a^2+\norm{\sum_{j=1}^{\ell_i} c_{ij}\pi (\psi_j^i - \psi_{j,\text{ms}}^i)}_s^2  \right) \\
& \leq C(k+1)^d E \sum_{i=1}^N \sum_{j=1}^{\ell_i} (c_{ij})^2 \left(\norm{\psi_j^i}_a^2+\norm{\pi\psi_j^i}_s^2 \right).
\end{split}
\end{eqnarray*}
Choosing test function $v = \psi^i_j$ in \eqref{eqn:min_relax}, 
we obtain that $\norm{\psi^i_j}_a^2+\norm{\pi\psi^i_j}^2_s\leq \norm{\phi^i_j}^2_s $. Therefore,
\begin{eqnarray}
\norm{\mbf{u}_{\text{glo}} - \mbf{v}}_a^2\lesssim E(k+1)^d \sum_{i=1}^N \sum_{j=1}^{\ell_i} (c_{ij})^2 \norm{\phi^i_j}_s^2 = E(k+1)^d \sum_{i=1}^N \sum_{j=1}^{\ell_i} (c_{ij})^2.
\end{eqnarray}
Next, we estimate the term $\displaystyle{\sum_{i=1}^N \sum_{j=1}^{\ell_i} (c_{ij})^2}$. 
Note that 
$\displaystyle{\pi(\mbf{u}_{\text{glo}}) = \sum_{i=1}^N  \sum_{j=1}^{\ell_i} c_{ij} \pi(\psi_j^i)}$. 
Using the variational formulation \eqref{eqn:min_relax} with test function $v = \psi_j^i$, we obtain 
$$b_{\ell k} := s\left (\pi(\mbf{u}_{\text{glo}}) , \phi_k^\ell \right ) = \sum_{i=1}^N  \sum_{j=1}^{\ell_i} c_{ij} s(\pi(\psi_j^i), \phi_k^\ell ) = \sum_{i=1}^N  \sum_{j=1}^{\ell_i} c_{ij} \left ( \underbrace{a(\psi_j^i, \psi_k^\ell) + s(\pi(\psi_j^i), \pi(\psi_k^\ell))}_{=: a_{ij, \ell k}} \right ).$$
If we denote $\mbf{b} = \left (b_{\ell k} \right ) \in \mathbb{R}^{\mathcal{N}}$ and $\mbf{c} = \left ( c_{ij} \right ) \in \mathbb{R}^{\mathcal{N}}$ with $\mathcal{N} :=\displaystyle{\sum_{i=1}^N \ell_i}$, then we have 
$$ \mbf{b} = A \mbf{c} \quad \text{and} \quad \norm{\mbf{c}}_2 \leq \norm{A^{-1}}_2 \norm{\mbf{b}}_2,$$
where $A := \left ( a_{ij,\ell k} \right ) \in \mathbb{R}^{\mathcal{N} \times \mathcal{N}}$ and $\norm{\cdot}_2$ denotes the standard Euclidean norm for vectors in $\mathbb{R}^{\mathcal{N}}$ and its induced matrix norm in $\mathbb{R}^{\mathcal{N} \times \mathcal{N}}$.
By the definition of $\pi: \mbf{V} \to V_{\text{aux}}$, we have 
$$ \pi(\mbf{u}_{\text{glo}}) = \pi\left (\pi (\mbf{u}_{\text{glo}}) \right ) = \sum_{i=1}^N \sum_{j=1}^{\ell_i} s(\pi(\mbf{u}_{\text{glo}}), \phi_j^i) \phi_j^i = \sum_{i=1}^N \sum_{j=1}^{\ell_i} b_{ij} \phi_j^i.$$
Thus, we have $\norm{\mbf{b}}_2 = \norm{\pi(\mbf{u}_{\text{glo}})}_s$. 
We define $\phi:=\displaystyle{\sum_{i=1}^N \sum_{j=1}^{\ell_i} c_{ij}\phi^i_j}$. Note that $\norm{\phi}_s = \norm{\mbf{c}}_2$. Consequently, by Lemma \ref{lemma:properties}, there exists a function $z \in \mbf{V}_0^{\text{div}}$ such that $\pi (z) = \phi$ and $\norm{z}_a^2 \leq D \norm{\phi}_s^2$. 
Since the multiscale basis $\psi^i_j$ satisfies \eqref{eqn:min_relax} and $\mbf{u}_{\text{glo}} $ is a linear combination of $\psi^i_j$'s, we have 
\begin{eqnarray}
a(\mbf{u}_{\text{glo}}, v) + s(\pi(\mbf{u}_{\text{glo}}), \pi v ) = s(\phi, \pi v) \quad \tforall v\in \mbf{V}_0^{\text{div}}(\Omega^\epsilon).
\label{u_glo:variabtional}
\end{eqnarray}
Picking $v=z$ in \eqref{u_glo:variabtional}, we arrive at
\begin{eqnarray*}
\begin{split}
\norm{\phi}^2_s&=a(\mbf{u}_{\text{glo}}, z) + s(\pi(\mbf{u}_{\text{glo}}), \pi z ) \leq \norm{\mbf{u}_{\text{glo}}}_a\cdot D^{1/2}\norm{\phi}_s+\norm{\pi\mbf{u}_{\text{glo}}}_s\cdot\norm{\phi}_s\\
&\leq (1+D)^{1/2}\norm{\phi}_s \left (\norm{ \mbf{u}_{\text{glo}}}^2_a+\norm{\pi\mbf{u}_{\text{glo}}}^2_s \right )^{1/2}.
\end{split}
\end{eqnarray*}
Therefore, we have 
$$ \norm{\mbf{c}}_2^2 = \norm{\phi}_s^2 \leq (1+D) \left (\norm{ \mbf{u}_{\text{glo}}}^2_a+\norm{\pi\mbf{u}_{\text{glo}}}^2_s \right ) = (1+D) \mbf{c}^T A \mbf{c}.$$

From the above, we see that the largest eigenvalue of $A^{-1}$ is bounded by $(1+D)$ and we have the following estimate 
$$ \norm{\mbf{c}}_2 \leq (1+D) \norm{\mbf{b}}_2 = (1+D) \norm{\mbf{u}_{\text{glo}}}_s.$$

As a result, we have 
$$ \norm{\mbf{u}_{\text{glo}} - \mbf{v}}_a^2 \leq (k+1)^d E (1+D)^2 \norm{\mbf{u}_{\text{glo}}}_s^2.$$
It remains to estimate the term $\norm{\mbf{u}_{\text{glo}}}_s$. In particular, we have 
$$ \norm{\mbf{u}_{\text{glo}}}_s^2 \lesssim \max \{ \tilde \kappa \} \norm{\mbf{u}_{\text{glo}}}_a^2 = \max \{ \tilde \kappa \} \left \langle \mbf{f} , \mbf{u}_{\text{glo}} \right \rangle \leq  \max \{ \tilde \kappa \} \norm{\tilde \kappa^{-1/2} \mbf{f} } \norm{\mbf{u}_{\text{glo}}}_s.$$
Therefore, we have 
$$
\norm{\mathbf{u}-\mathbf{u}_{\text{ms}}}_a \lesssim \Lambda^{-1} \norm{\tilde \kappa^{-1/2} \mathbf{f}} + \max \{ \tilde \kappa \} (k+1)^{d/2} E^{1/2} (1+D)  \norm{\tilde \kappa^{-1/2} \mbf{f}}.
$$
If we take $k = O(\log(H^{-1}))$ and assume that $\{ \chi_i^{\text{ms}} \}_{i=1}^{N_c}$ is a set of bilinear partition of unity, then we have $$ \norm{\mbf{u} - \mbf{u}_{\text{ms}}}_a \lesssim H \Lambda^{-1} \norm{\mbf{f}}.$$
This completes the proof. 
\end{proof}

\section{Conclusion} \label{sec:conclusion}
In this work, we have proposed and analyzed the constraint energy minimizing generalized multiscale finite element method for solving the incompressible Stokes flows in perforated domain. 
The proposed method started with a local spectral decomposition of the continuous Sobolev space. 
Based on the concepts of constraint energy minimization and oversampling, we construct divergence-free multiscale basis functions for displacement variable satisfying the property of least energy. The pressure variable is thus recovered on the coarse-grid based on the multiscale approximation of displacement. 
The method is shown to have spectral convergence with error bound proportional to the coarse mesh size. 

\section*{Acknowledgement}

The research of Eric Chung is partially supported by the Hong Kong RGC General Research Fund (Project numbers 14304217 and 14302018) and CUHK Faculty of Science Direct Grant 2019-20.

\bibliographystyle{abbrv}
\bibliography{references}

\begin{thebibliography}{10}

\bibitem{bm97}
I.~Babu{\v{s}}ka and J.~M. Melenk.
\newblock The partition of unity method.
\newblock {\em Int. J. Numer. Meth. Engrg.}, 40:727--758, 1997.

\bibitem{Stochastic2004}
A.~Bourgeat and A.~Piatnitski.
\newblock Approximations of effective coefficients in stochastic
  homogenization.
\newblock {\em Ann. Inst. H. Poincar\'{e} Probab. Statist.}, 40(2):153--165,
  2004.

\bibitem{brown2016multiscale}
D.~L. Brown and D.~Peterseim.
\newblock A multiscale method for porous microstructures.
\newblock {\em Multiscale Modeling \& Simulation}, 14(3):1123--1152, 2016.

\bibitem{chung2016adaptive}
E.~Chung, Y.~Efendiev, and T.~Y. Hou.
\newblock Adaptive multiscale model reduction with generalized multiscale
  finite element methods.
\newblock {\em Journal of Computational Physics}, 320:69--95, 2016.

\bibitem{CEM_elliptic}
E.~Chung, Y.~Efendiev, and W.~T. Leung.
\newblock Constraint energy minimizing generalized multiscale finite element
  method.
\newblock {\em Comput. Methods Appl. Mech. Engrg.}, 339:298--319, 2018.

\bibitem{chung2018cemmixed}
E.~Chung, Y.~Efendiev, and W.-T. Leung.
\newblock Constraint energy minimizing generalized multiscale finite element
  method in the mixed formulation.
\newblock {\em Computational Geosciences}, 22(3):677--693, 2018.

\bibitem{chung2017onlinep}
E.~Chung, Y.~Efendiev, W.~T. Leung, M.~Vasilyeva, and Y.~Wang.
\newblock Online adaptive local multiscale model reduction for heterogeneous
  problems in perforated domains.
\newblock {\em Applicable Analysis}, 96(12):2002--2031, 2017.

\bibitem{chung2016mixed}
E.~Chung, W.~T. Leung, and M.~Vasilyeva.
\newblock Mixed gmsfem for second order elliptic problem in perforated domains.
\newblock {\em Journal of Computational and Applied Mathematics}, 304:84--99,
  2016.

\bibitem{chung2020computational}
E.~Chung and S.-M. Pun.
\newblock Computational multiscale methods for first-order wave equation using
  mixed cem-gmsfem.
\newblock {\em Journal of Computational Physics}, page 109359, 2020.

\bibitem{Chung2017Stokes}
E.~Chung, M.~Vasilyeva, and Y.~Wang.
\newblock A conservative local multiscale model reduction technique for
  {S}tokes flows in heterogeneous perforated domains.
\newblock {\em J. Comput. Appl. Math.}, 321:389--405, 2017.

\bibitem{chung2018multiscaleflow}
E.~T. Chung, W.~T. Leung, M.~Vasilyeva, and Y.~Wang.
\newblock Multiscale model reduction for transport and flow problems in
  perforated domains.
\newblock {\em Journal of Computational and Applied Mathematics}, 330:519--535,
  2018.

\bibitem{ciarlet2013linear}
P.~G. Ciarlet.
\newblock {\em Linear and {N}onlinear {F}unctional {A}nalysis with
  {A}pplications}.
\newblock Society for Industrial and Applied Mathematics, 2013.

\bibitem{efendiev2013generalized}
Y.~Efendiev, J.~Galvis, and T.~Y. Hou.
\newblock Generalized multiscale finite element methods ({GM}s{FEM}).
\newblock {\em Journal of Computational Physics}, 251:116--135, 2013.

\bibitem{YalchinHouMultiscale}
Y.~Efendiev and T.~Y. Hou.
\newblock {\em Multiscale finite element methods}, volume~4 of {\em Surveys and
  Tutorials in the Applied Mathematical Sciences}.
\newblock Springer, New York, 2009.
\newblock Theory and applications.

\bibitem{efendiev2009multiscale}
Y.~Efendiev and T.~Y. Hou.
\newblock {\em Multiscale finite element methods: theory and applications},
  volume~4.
\newblock Springer Science \& Business Media, 2009.

\bibitem{engwer2019efficient}
C.~Engwer, P.~Henning, A.~M{\aa}lqvist, and D.~Peterseim.
\newblock Efficient implementation of the localized orthogonal decomposition
  method.
\newblock {\em Computer Methods in Applied Mechanics and Engineering},
  350:123--153, 2019.

\bibitem{Guermond2004book}
A.~Ern and J.-L. Guermond.
\newblock {\em Theory and practice of finite elements}, volume 159.
\newblock Springer Science \& Business Media, 2013.

\bibitem{feng2018crouzeix}
Q.~Feng, G.~Allaire, and M.~Puscas.
\newblock Crouzeix-raviart multiscale finite element method for {S}tokes flows
  in heterogeneous media.
\newblock In R.~Owen, R.~de~Borst, J.~Reese, and C.~Pearce, editors, {\em
  Proceedings of the 6th. European Conference on Computational Mechanics
  (Solids, Structures and Coupled Problems) and 7th. European Conference on
  Computational Fluid Dynamics}, pages 818--827. 2018.

\bibitem{fu2019computational}
S.~Fu, R.~Altmann, E.~Chung, R.~Maier, D.~Peterseim, and S.-M. Pun.
\newblock Computational multiscale methods for linear poroelasticity with high
  contrast.
\newblock {\em Journal of Computational Physics}, 395:286--297, 2019.

\bibitem{CEM2020Elasticity}
S.~Fu, E.~Chung, and T.~Mai.
\newblock Constraint energy minimizing generalized multiscale finite element
  method for nonlinear poroelasticity and elasticity.
\newblock {\em J. Comput. Phys.}, 417:109569, 2020.

\bibitem{henning2009heterogeneous}
P.~Henning and M.~Ohlberger.
\newblock The heterogeneous multiscale finite element method for elliptic
  homogenization problems in perforated domains.
\newblock {\em Numerische Mathematik}, 113(4):601--629, 2009.

\bibitem{Homo2018Stokes}
M.~Hillairet.
\newblock On the homogenization of the {S}tokes problem in a perforated domain.
\newblock {\em Arch. Ration. Mech. Anal.}, 230(3):1179--1228, 2018.

\bibitem{hornung1996homogenization}
U.~Hornung, editor.
\newblock {\em Homogenization and {P}orous {M}edia}, volume~6 of {\em
  Interdisciplinary Applied Mathematics}.
\newblock Springer-Verlag New York, 1997.

\bibitem{msfemstokesp2}
G.~Jankowiak and A.~Lozinski.
\newblock Non-conforming multiscale finite element method for {S}tokes flows in
  heterogeneous media. {P}art {II}: Error estimates for periodic
  microstructure.
\newblock {\em arXiv preprint arXiv:1802.04389}, 2018.

\bibitem{le2014msfem}
C.~Le~Bris, F.~Legoll, and A.~Lozinski.
\newblock An {M}s{FEM} type approach for perforated domains.
\newblock {\em Multiscale Modeling \& Simulation}, 12(3):1046--1077, 2014.

\bibitem{Peterseim2014}
A.~M{\aa}lqvist and D.~Peterseim.
\newblock Localization of elliptic multiscale problems.
\newblock {\em Mathematics of Computation}, 83(290):2583--2603, 2014.

\bibitem{msfemstokesp1}
B.~P. Muljadi, J.~Narski, A.~Lozinski, and P.~Degond.
\newblock Nonconforming multiscale finite element method for stokes flows in
  heterogeneous media. {P}art {I}: methodologies and numerical experiments.
\newblock {\em Multiscale Modeling \& Simulation}, 13(4):1146--1172, 2015.

\bibitem{CEM2019Flow}
M.~Vasilyeva, E.~Chung, Y.~Efendiev, and J.~Kim.
\newblock Constrained energy minimization based upscaling for coupled flow and
  mechanics.
\newblock {\em J. Comput. Phys.}, 376:660--674, 2019.

\bibitem{Yosifian1997Homo}
G.~A. Yosifian.
\newblock On some homogenization problems in perforated domains with nonlinear
  boundary conditions.
\newblock {\em Appl. Anal.}, 65(3-4):257--288, 1997.

\bibitem{Micromechanics}
T.~I. Zohdi and P.~Wriggers.
\newblock {\em An introduction to computational micromechanics}, volume~20 of
  {\em Lecture Notes in Applied and Computational Mechanics}.
\newblock Springer-Verlag, Berlin, 2008.
\newblock Corrected second printing of the 2005 original.

\end{thebibliography}
\end{document}